\documentclass[sts]{imsart}

\usepackage{geometry}
\usepackage{amssymb,amsbsy,amsmath,amscd,amsthm,amsfonts,MnSymbol}
\usepackage{cite}
\usepackage[utf8]{inputenc}
\usepackage{enumerate,hyperref}
\usepackage{dsfont}
\usepackage{graphicx}
\usepackage{tikz}

\newtheorem{theorem}{Theorem}

\newtheorem{assumption}{Assumption}

\newtheorem{lemma}{Lemma}
\newtheorem{remark}{Remark}

\newcommand{\R}{\mathbb R}

\newcommand{\E}{\mathbb E}

\newcommand{\KL}{\textrm{KL}}
\newcommand{\TV}{\textrm{TV}}

\newcommand*\diff{\mathop{}\!\mathrm{d}}

\newcommand{\inter}{\mathrm{int}}

\newcommand{\relint}{\mathrm{relint}}

\newcommand{\pn}{^{(n)}}

\newcommand{\DS}{\displaystyle}

\DeclareMathOperator*{\argmin}{argmin}

\begin{document}

\begin{frontmatter}

\title{Bernstein-von Mises theorem for log-concave posteriors}
\runtitle{Log-concave BvM}

\author{Victor-Emmanuel Brunel \thanks{CREST-ENSAE, victor.emmanuel.brunel@ensae.fr}}

\runauthor{V.-E. Brunel}

\setattribute{abstractname}{skip} {{\bf Abstract:} } 
\begin{abstract}
We prove new, general versions of Bernstein-von Mises theorem for both well-specified and misspecified models when the log-likelihood is concave in the parameter and the prior distribution is log-concave. Unlike classical versions of Bernstein-von Mises theorem, our versions do not require technical smoothness assumptions, and they solely rely on convex analysis.
\end{abstract}


\end{frontmatter}

\section{Introduction}

\subsection{Framework} \label{sec:framework}
Let $(E,\mathcal E)$ be a measurable space and consider a family $(P_\theta)_{\theta\in \Theta_0}$ of probability distributions on $(E,\mathcal E)$ that satisfy the following:
\begin{itemize}
	\item The set $\Theta_0$ is a non-empty, open, convex subset of $\R^d$ for some $d\geq 1$;
	\item The family $(P_\theta)_{\theta\in \Theta_0}$ is dominated -- that is, there exists a $\sigma$-finite measure $\mu$ on $(E,\mathcal E)$ such that $P_\theta$ has a density $f_\theta$ with respect to $\mu$, for all $\theta\in\Theta_0$;
	\item For all $\theta\in \Theta_0$ and for $\mu$-almost all $x\in E$, $f_\theta(x)>0$;
	\item For $\mu$-almost all $x\in E$, the map $\theta\in\Theta_0\mapsto \phi(x,\theta):=-\log f_\theta(x)$ is convex.
\end{itemize}

Consider a sequence $X_1,X_2,\ldots$ of independent, identically distributed (i.i.d.) random variables on $E$, whose distribution is given by $P_{\bar\theta}$ for some $\bar\theta\in\Theta_0$. Let $\pi$ be a log-concave distribution on $\R^d$ whose support is a (closed) convex set $\Theta\subseteq \Theta_0$. Throughout this work, we assume that $\Theta$ has non-empty interior. By \cite[Theorem 3.2]{borell1975convex}, this implies that $\pi$ has a density with respect to the Lebesgue measure of the form $e^{-V}$ for some convex function $V:\R^d\to\R\cup\{\infty\}$ with $V(\theta)<\infty$ for all $\theta\in\Theta$ and $V(\theta)=\infty$ for all $\theta\notin \Theta$.

In a Bayesian framework, if $\pi$ is seen as a prior distribution, then for each $n\geq 1$, the posterior distribution, denoted by $\tilde Q_n$, is the (random) log-concave probability measure on $\R^d$ whose density with respect to the Lebesgue measure is given by:
\begin{equation} \label{eqn:posterior}
	\tilde q_n(\theta)=\tilde Z_ne^{-n\Phi_n(\theta)-V(\theta)}\mathds 1_{\theta\in \Theta}, \quad \forall \theta\in\Theta_0
\end{equation}
where $\Phi_n(\theta)=n^{-1}\sum_{i=1}^n \phi(X_i,\theta)$, $\theta\in\Theta_0$, is the negative log-likelihood and $\tilde Z_n$ is a (random) normalizing constant (see Theorem~\ref{THMPROPER} below). 

This work is concerned with the asymptotic posterior distribution of rescaled versions of $\theta$ given $X_1,\ldots,X_n$, where the appropriate rescalings will depend on natural, geometric conditions. We assume that the map $V$ is well-behaved, that is, $V(\theta^*)\neq \infty$ and $V$ is differentiable at $\theta^*$ (along the tangent cone $T$). This is only a stringent condition when $\theta^*$ is on the boundary of $\Theta$ since it excludes the case when the prior density $e^{-V}$ can vanish as $\theta$ approaches $\theta^*$. While this is an interesting case, it would require further technical arguments, which we leave for future work. In the rest of this work, for simplicity, we assume that $V=\delta_\Theta$, the indicatrix of $\Theta$, that is, that the (possibly improper) prior density is $\mathds 1_{\cdot\in\Theta}$. We let the reader check that given the assumption above, this is merely a presentation simplification, rather than a technical one.

In what follows, we denote by 
$$\Phi(\theta)=\E[\phi(X_1,\theta)]=\int_{E}\phi(x,\theta)e^{-\phi(x,\theta^*)}\diff\mu(x) \in \R\cup\{\infty\}, \quad \theta\in\Theta_0.$$
The function $\Phi$ is convex and for simplicity, we assume that $\Phi(\theta)<\infty$ for all $\theta\in\Theta_0$. Note that $\Phi(\theta)-\Phi(\bar\theta)=\KL(P_{\bar\theta}\|P_\theta)$, where $\KL$ stands for the Kullback-Leibler divergence. Throughout this work, we denote by $\theta^*=\argmin_{\theta\in\Theta}\Phi(\theta)$ and we assume that $\Phi$ is twice continuoulsy differentiable in a neighborhood of $\theta^*$ and that $\nabla^2\Phi(\theta^*)$ is positive definite. When $\theta^*\in\inter(\Theta)$, then $\bar\theta=\theta^*$ and we say that the Bayesian model is \textit{well-speicified}. When $\theta^*$ is on the boundary of $\Theta$ and $\nabla\Phi(\theta)=0$, then again, $\bar\theta=\theta^*$. In that case, we say that the Bayesian model is \textit{nearly misspecified}. Finally, when $\theta^*$ is on the boundary of $\Theta$ and $\nabla\Phi(\theta^*)\neq 0$, we say that the Bayesian model is \textit{misspecified}: The support of the prior distribution $\pi$ fails at containing $\bar\theta$. In that case, $\theta^*$ is simply the value of the parameter $\theta\in\Theta$ for which $P_\theta$ is as close to $P_{\bar\theta}$ -- the distribution of the $X_i$'s -- as possible. In other words, $P_{\theta^*}$ is the best KL approximation of the data distribution within the model $(P_{\theta})_{\theta\in\Theta}$. In that case, the first order optimality condition for $\theta^*$ states that $-\nabla\Phi(\theta^*)$ must be in the normal cone of $\Theta$ at $\theta^*$. 

Note that positive definiteness of $\nabla^2\Phi(\theta^*)$ implies that $\theta^*$ is identifiable, i.e., $P_\theta=P_{\theta^*}$ can only happen if $\theta=\theta^*$, and hence, that $\theta^*$ is the unique minimizer of $\Phi$ on $\Theta$. In the well-specified and nearly misspecified cases, the matrix $\nabla^2\Phi(\theta^*)$ is simply the Fisher information matrix, which is therefore positive definite. Moreover, convexity of $\phi$ in $\Theta$ and the fact that $\Phi$ is twice continuously differentiable in a neighborhood of $\theta^*$ yield that the posterior distribution is proper, validating the definition of $\tilde Z_n$ in \eqref{eqn:posterior}, as shown by the following theorem.

\begin{theorem} \label{THMPROPER}
	Under the assumptions above, with probability $1$, the map $\theta\in\Theta_0\mapsto e^{-n\Phi_n(\theta)}\mathds 1_{\theta\in\Theta}$ has a finite integral with respect to the Lebesgue measure. 
\end{theorem}

The proof of this theorem, which also is a sanity validation of our whole work, is provided in the appendix, in Section~\ref{sec:proof_thm0}.

The following will be needed in this work. By \cite[Theorem 2]{brunel2025asymptotics}, one can define i.i.d. random variables $U_1,\ldots,U_n$ in $\R^d$ with the property that, for each $i=1,\ldots,n$, $U_i$ is almost surely a subgradient of $\phi(X_i,\cdot)$ at $\theta^*$. Moreover, \cite[Theorem 3]{brunel2025asymptotics} yields that the $U_i$'s are integrable, with $\E[U_i]=\nabla\Phi(\theta^*)$. We further make the following assumption on the $U_i$'s, which would be granted automatically by \cite[Theorem 3 and Remark 2]{brunel2025asymptotics}, if $\phi(X_1,\theta)$ were square integrable for all $\theta$ in an arbitrarily small neighborhood of $\theta^*$.

\begin{assumption} \label{ass:subgradients_integr}
	$\E[\|U_1\|^2]<\infty$.
\end{assumption}

Our work is organized as follows. In Section~\ref{sec:case1}, we cover the well-specified case. In Section~\ref{sec:case2}, we deal with the nearly misspecified case and Section~\ref{sec:case3} is devoted to the misspecified case. The misspecified case is fundamentally different from the previous two, and requires more geometric tools. We will see that the appropriate scaling of $\theta$ differs from one case to the other. Only in the last case, we will assume that $\Theta$ can be written as a set of finitely many constraints defined by twice differentiable, convex functions. In the conclusion (\ref{sec:conclusion}), we briefly discuss why no assumption on $\Theta$ is needed to cover the first two cases, while some assumptions on the boundary of $\Theta$, at least in the vicinity of $\theta^*$, seem necessary in the misspecified case.

\subsection{Related work and contributions}

In a Bayesian context, Bernstein-von Mises theorem gives the limiting shape -- typically Gaussian -- of a posterior distribution as the sample sizes goes to infinity \cite{le1953some,walker1969asymptotic,dawid1970limiting}.
Common assumptions are technical, including smoothness of the log-likelihood function as well as uniform domination conditions. Extensions of Bernstein-von Mises to non-regular models have also been explored for exponential families using likelihood geometry \cite{goehle2024bayesian}. 

Bernstein-von Mises theorem has also been extended to non-parametric and high-dimensional setups \cite{castillo2013nonparametric} (see \cite{castillo2024bayesian} for a recent review). Closely related to our work, a line of work is concerned with regular, misspecified parametric models, that is, when the support of the prior distribution does not contain the true parameter of the data distribution \cite{kleijn2012bernstein,bochkina2014bernstein,bochkina2019bernstein}. However, these works also rely on technical conditions on the log-likelihood function and they typically assume a very simple form for $\Theta$, e.g., $\Theta=(\R_+)^d$. Laplace approximations for error quantification Bernstein-von Mises theorem have also been investigated in a recent line of work \cite{katsevich2025improved,spokoiny2025inexact}.

In this work, we are concerned with extensions of the classical Bernstein-von Mises theorem where we only assume log-concavity of the posterior distribution. Our goal is to show that, under such convexity condition -- with additional second order differentiability of the Kullback-Leibler divergence around the $\theta^*$ -- no smoothness/domination assumption is needed. 
We consider both well-specified and misspecified models, where the support of the log-concave prior distribution may not contain the true parameter. Unlike recent works \cite{bochkina2014bernstein}, we consider a large class of supports, including smooth convex sets and polytopes. In particular, we show that the limiting shape of the posterior distribution is supported on the second-order tangent set of $\Theta$ at $\theta^*$ along the direction of $-\nabla\Phi(\theta^*)$. To the best of our knowledge, our results are new even in the well-specified case. 

A parallel should be drawn with the asymptotic analysis of maximum likelihood estimators (MLE), where classical results derive the asymptotic distribution of a rescaled version of the MLE under technical smoothness/domination assumptions on the log-likelihood in well-specified models \cite{lecam1970assumptions,van1996weak,van2000asymptotic} as well as misspecified models where a constrained MLE is studied \cite{geyer1994asymptotics}. The assumptions made in these works are similar to these made for classical Bernstein-von Mises theorem. In a recent work, \cite{brunel2025asymptotics} derives the asymptotic distribution of rescaled versions of constrained MLE (more generally, constrained $M$-estimators), under minimal convexity assumptions on the loss function and the constraint set, allowing to consider non-smooth models. Our work builds on similar ideas as \cite{brunel2025asymptotics}, borrowed from convex analysis and convex optimization.

\subsection{Notation}

The interior of a set $A\subseteq \R^d$ is denoted by $\inter(A)$. If $A$ is a convex set, its relative interior is denoted by $\relint(A)$: This is the interior of $A$ relative to the topology of the affine span of $A$. 

Given two random variables $U$ and $V$, we denote by $P_{U|V}$ the conditional distribution of $U$ given $V$. 

The distance in total variation between two Borel probability measures $P$ and $Q$ on $\R^d$ is defined as 
$$\TV(P,Q)=\sup_{B\in\mathcal B(\R^d)}|P(B)-Q(B)|=\int_{\R^d}|p(x)-q(x)|\diff\mu(x)$$
where $\mathcal B(\R^d)$ is the Borel $\sigma$-algebra of $\R^d$ and $p$ and $q$ are the respective densities of $P$ and $Q$ with respect to a common dominating measure $\mu$ ($\mu$ will be the Lebesgue measure in this work).

In what follows, we denote by $T$ the tangent cone of $\Theta$ at $\theta^*$, that is, 
$\DS T=\bigcup_{p\geq 1} p(\Theta-\theta^*)$,
and by $C$ its closure, called the support cone of $\Theta$ at $\theta^*$. The normal cone of $\Theta$ at $\theta^*$ is denoted by $N$, and it is the polar of $C$: $N=\{t\in\R^d:t^\top (\theta-\theta^*)\leq 0,\forall \theta\in\Theta\}=\{t\in\R^d=t^\top u\leq 0, \forall u\in C\}$.

\section{Well-specified case} \label{sec:case1}

In this section, we assume that $\theta^*$ is an interior point of $\Theta$. As a consequence, $C=\R^d$ and the first order condition for $\theta^*$ simply reads as $\nabla\Phi(\theta^*)=0$. In particular, this means that the constrained and unconstrained minimizers of $\Phi$ coincide. In that case, we consider the following scaling of $\theta$: $t=\sqrt n(\theta-\theta^*)$ and our objective is to derive the asymptotic behavior of the posterior distribution $Q_n$ of $t$ given $X_1,\ldots,X_n$. The main result of this section is the following asymptotic normality of the posterior distribution, which recovers the classical Bernstein-von Mises theorem in smooth setups.Denote by $Y_n=\frac{1}{\sqrt n}\sum_{i=1}^nU_i$. 

\begin{theorem} \label{thm:main_case1}
Let Assumption~\ref{ass:subgradients_integr} hold and assume that $\theta^*\in\inter(\Theta)$. Then,
	$$\TV(Q_n,\mathcal N_d(-\nabla^2\Phi(\theta^*)^{-1}Y_n,\nabla^2\Phi(\theta^*)^{-1}))\xrightarrow[n\to\infty]{} 0$$
in probability.
\end{theorem}

Note that $\sqrt n(\hat \theta_n-\theta^*)+\nabla^2\Phi(\theta^*)Y_n\xrightarrow[n\to\infty]{}0$ in probability, as shown in the proof of Theorem~7 in \cite{brunel2025asymptotics}, where $\hat \theta_n$ is the constrained\footnote{since $\theta^*\in\inter(\Theta)$, the constrained and unconstrained maximum likelihood estimators coincide for sufficiently large $n$, with probability $1$, as a consequence of \cite[Theorem 4]{brunel2025asymptotics}.} maximum likelihood estimator. Hence, one can easily check that the above statement can also be written, equivalently, as:
$$\TV(Q_n,\mathcal N_d(\sqrt n(\hat \theta_n-\theta^*),\nabla^2\Phi(\theta^*)^{-1}))\xrightarrow[n\to\infty]{} 0$$
in probability or, also equivalently, as:
$$\TV(P_{\sqrt n(\theta-\hat\theta_n)|X_1,\ldots,X_n},\mathcal N_d(0,\nabla^2\Phi(\theta^*)^{-1}))\xrightarrow[n\to\infty]{} 0$$
in probability (note that $\hat\theta_n$ is a measurable function of $X_1,\ldots,X_n$). 

In particular, when $\Theta_0=\Theta=\R^d$, we retrieve the most classical version of Bernstein-von Mises theorem, under very minimal assumptions: We do not assume any smoothness on the negative log-likelihood function $\phi$, apart from the existence of the second order derivative of its expectation. We believe that this result is new. 

\begin{proof}[Proof of Theorem~\ref{thm:main_case1}]

First, without loss of generality -- and for the sake of simplicity -- let us assume that $\nabla^2\Phi(\theta^*)=I_d$. Indeed, we could consider the rescaling $\tilde t:= \nabla^2\Phi(\theta^*)^{-1/2}t$ instead of $t$, which would require a similar analysis.

Uusing \eqref{eqn:posterior}, note that $Q_n$ has a density with respect to the Lebesgue measure, which is proportional to $t\in\R^d\mapsto e^{-n\left(\Phi_n(\theta^*+t/\sqrt n)-\Phi(\theta^*)\right)}\mathds 1_{t\in \sqrt n(\Theta-\theta^*)}$.

For $t\in\R^d$, denote by 
$$G_n(t)=n\left(\Phi_n(\theta^*+t/\sqrt n)-\Phi(\theta^*)\right)-t^\top Y_n.$$
Then, the density $q_n$ of $Q_n$ can be written as
\begin{equation} \label{eqn:densityqn}
	q_n(t)=\frac{Z_n}{(2\pi)^{d/2}}\exp\left(-G_n(t)-t^\top Y_n-\|Y_n\|^2/2\right) \mathds 1_{t\in\sqrt n(\Theta-\theta^*)}, \quad \forall t\in\R^d,
\end{equation}
for some (random) normalizing constant $Z_n>0$. Our goal is to show that $Z_n\xrightarrow[n\to\infty]{}1$ in probability. 

Fix some $\varepsilon>0$. Under Assumption~\ref{ass:subgradients_integr}, the central limit theorem yields that the sequence $(Y_n)_{n\geq 1}$ is tight. Hence, one can find $R>0$ such that for all $n\geq 1$, 
\begin{equation} \label{eqn:tighness1}
	P(\|Y_n\|\leq R)\geq 1-\varepsilon.
\end{equation}

The following lemma, proven in the proof of \cite[Theorem 7]{brunel2025asymptotics} is fundamental for our analysis.
\begin{lemma} \label{lem:UnifConv}
	For all compact subsets $K\subseteq\R^d$, 
	$$\sup_{t\in K}\left|G_n(t)-\frac{\|t\|^2}{2}\right|\xrightarrow[n\to\infty]{} 0$$
in probability.
\end{lemma}

As a consequence, by taking $K=B(0,4R)$, we have that for all sufficiently large $n\geq 1$, with probability at least $1-\varepsilon$, 
\begin{equation} \label{eqn:controlGn} \sup_{\|t\|\leq 4R}|G_n(t)-\|t\|^2/2|\leq \varepsilon.
\end{equation}

Assume for now that $\|Y_n\|\leq R$ and $\sup_{\|t\|\leq 4R}|G_n(t)-\|t\|^2/2|\leq\varepsilon$ -- by the union bound together with \eqref{eqn:tighness1} and \eqref{eqn:controlGn}, this event occurs with probability at least $1-2\varepsilon$ for all sufficiently large $n$.

Convexity of $G_n$ implies that for all $t\in\R^d$ with $\|t\|> 4R$ (writing $4Rt/\|t\|$ as a convex combination of $t$ and $0$), 
\begin{align} 
	G_n(t) & \geq \frac{\|t\|}{4R}G_n(4Rt/\|t\|) - \frac{\|t\|-4R}{4R}G_n(0) \geq \frac{\|t\|}{4R}(8R^2-\varepsilon)-\frac{\|t\|-4R}{4R}\varepsilon \nonumber \\ 
	& = (2R-\varepsilon/(2R))\|t\|+\varepsilon \geq (2R-\varepsilon/(2R))\|t\|. \label{eqn:controlFn}
\end{align}
Therefore, for all $t\in\R^d$,
\begin{align} 
q_n(t) & \leq \frac{Z_n}{(2\pi)^{d/2}}\left(e^{\varepsilon-\frac{1}{2}\|t+Y_n\|^2}\mathds 1_{\|t\|\leq 4R}+e^{-\varepsilon-(2R-\varepsilon/(2R))\|t\|-t^\top Y_n-\|Y_n\|^2/2}\mathds 1_{\|t\|>4R}\right) \nonumber \\
& = \frac{Z_n}{(2\pi)^{d/2}}\left(A+B\right) \label{eqn:step1_1}
\end{align}
where, for the second term (corresponding to $\|t\|>2R$), we used \eqref{eqn:controlFn}.

We first control the $A$ term by plainly bounding the indicator by $1$, yielding that $A\leq e^{\varepsilon}e^{\frac{1}{2}\|t+Y_n\|^2}$. For the $B$ term, we use Cauchy-Schwarz inequality which yields that $-t^\top Y_n\leq R\|t\|$, so
\begin{equation*}
	B \leq e^{-\varepsilon-(R-\varepsilon/(2R))\|t\|}\mathds 1_{\|t\|>4R}.
\end{equation*}
Finally, integrating \eqref{eqn:step1_1} with respect to the Lebesgue measure on $\R^d$, we obtain
\begin{equation*}
	1 = \int_{\R^d}q_n(t)\diff t \leq \frac{Z_n}{(2\pi)^{d/2}}\left(e^\varepsilon (2\pi)^{d/2}+e^{-\varepsilon}\int_{\R^d\setminus B(0,4R)}e^{-(R-\varepsilon/(2R))\|t\|}\diff t\right).
\end{equation*}
Since the integral appearing on the right-hand side can be made arbitrarily small by choosing $R$ sufficiently large (recall that $R$ was defined from \eqref{eqn:tighness1} and could be taken arbitrarily large), we can choose $R$ such that the second integral is not larger than $(2\pi)^{d/2}\varepsilon$, yielding that $1\leq Z_n(e^{\varepsilon}+\varepsilon)$, that is, $Z_n\geq (e^{\varepsilon}+\varepsilon)^{-1}$. 

On the other hand, from \eqref{eqn:densityqn}, let us bound $q_n$ from below by first assuming that $n$ is large enough so $B(0,4R)\subseteq \sqrt n(\Theta-\theta^*)$ (this is possible because $\theta^*$ is in the interior of $\Theta$). This yields that $\mathds 1_{t\in\sqrt n(\Theta-\theta^*)}\geq \mathds 1_{\|t\|\leq 4R}$, so for all $t\in\R^d$,
\begin{equation*}
	q_n(t) \geq \frac{Z_n}{(2\pi)^{d/2}}e^{-\varepsilon-\frac{1}{2}\|t+Y_n\|^2} \mathds 1_{\|t\|\leq 4R} \geq \frac{Z_n}{(2\pi)^{d/2}}e^{-\varepsilon-\frac{1}{2}\|t+Y_n\|^2} \mathds 1_{\|t+Y_n\|\leq 3R} 
\end{equation*}
and hence, by integrating with respect to the Lebesgue measure, 
\begin{equation*}
	1 = \int_{\R^d}q_n(t)\diff t \geq \frac{Z_n}{(2\pi)^{d/2}}e^{-\varepsilon} \int_{B(0,3R)}e^{-\|u\|^2/2}\diff u.
\end{equation*}
Since the integral on the right-hand side goes to $(2\pi)^{d/2}$ as $R\to\infty$, let us choose $R$ sufficiently large so that integral is not smaller than $(2\pi)^{d/2}e^{-\varepsilon}$. This yields that $Z_n\leq e^{2\varepsilon}$. 

Finally, to recap, we have shown that, given any fixed $\varepsilon>0$, for sufficiently large $n$, $P((e^{\varepsilon}+\varepsilon)^{-1}\leq Z_n\leq e^{2\varepsilon})\geq 1-2\varepsilon$. Therefore, $Z_n\xrightarrow[n\to\infty]{} 1$ in probability, which is what we wanted to show. 

Now, for the last step of the proof, let us define $p_n:t\in\R^d\mapsto (2\pi)^{-d/2}e^{-\|t+Y_n\|^2/2}$ which is the density of $\mathcal N_d(-Y_n,I_d)$. Our goal is now to show that 
\begin{equation*}
	\int_{\R^d}|p_n(t)-q_n(t)|\diff t \xrightarrow[n\to\infty]{} 0
\end{equation*}
in probability. First, note that for all $x\in\R$, $|x|=2x_+-x$ where $x_+=\max(x,0)$. Applying this to $x=p_n(t)-q_n(t)$ for all $t\in\R^d$ in the above integral, and recalling that $p_n$ and $q_n$ are both densities, and hence have their integrals equal to $1$, we want to show that 
\begin{equation*} 
	\int_{\R^d}\left(1-Z_ne^{-(G_n(t)-\|t\|^2/2)}\mathds 1_{t\in\sqrt n(\Theta-\theta^*)}\right)_+p_n(t)\diff t \xrightarrow[n\to\infty]{} 0
\end{equation*}
in probability. Moreover, by subbaditivity of $x\in\R_+\mapsto x_+$, one can bound the above integral as follows 
\begin{align*} 
	\int_{\R^d} & \left(1-Z_ne^{-(G_n(t)-\|t\|^2/2)}\mathds 1_{t\in\sqrt n(\Theta-\theta^*)}\right)_+p_n(t)\diff t \\
	& \quad \quad = \int_{\R^d}\left(1-Z_n+Z_n-Z_ne^{-(G_n(t)-\|t\|^2/2)}\mathds 1_{t\in\sqrt n(\Theta-\theta^*)}\right)_+p_n(t)\diff t \\
	& \quad \quad \leq (1-Z_n)_+\int_{\R^d}p_n(t)\diff t+Z_n\int_{\R^d}\left(1-e^{-(G_n(t)-\|t\|^2/2)}\mathds 1_{t\in\sqrt n(\Theta-\theta^*)}\right)_+p_n(t)\diff t \\
	&  \quad \quad = (1-Z_n)_++Z_n\int_{\R^d}\left(1-e^{-(G_n(t)-\|t\|^2/2)}\mathds 1_{t\in\sqrt n(\Theta-\theta^*)}\right)_+p_n(t)\diff t.
\end{align*}
Since $Z_n\xrightarrow[n\to\infty]{}1$ in probability, it is therefore enough to show that 
\begin{equation} \label{eqn:goalTV}
	\int_{\R^d}\left(1-e^{-(G_n(t)-\|t\|^2/2)}\mathds 1_{t\in\sqrt n(\Theta-\theta^*)}\right)_+p_n(t)\diff t \xrightarrow[n\to\infty]{} 0
\end{equation}
in probability. Fix $\varepsilon>0$ and, as above, let $R$ satisfy \eqref{eqn:tighness1}. Assume again that $\|Y_n\|\leq R$ and that $\sup_{\|t\|\leq 4R}|G_n(t)-\|t\|^2/2|\leq\varepsilon$.  

First, decompose the integral in \eqref{eqn:goalTV} as the sum of one integral on $B(0,4R)$ and one integral on $\R^d\setminus B(0,4R)$. For the first one, we use the fact that the map $x\in\R\mapsto x_+$ is non-decreasing, so
\begin{align*}
	\int_{B(0,4R)} \left(1-e^{-(G_n(t)-\|t\|^2/2)}\mathds 1_{t\in\sqrt n(\Theta-\theta^*)}\right)_+p_n(t)\diff t & \leq \int_{\R^d}(1-e^{-\varepsilon}\mathds 1_{t\in\sqrt n(\Theta-\theta^*)})_+p_n(t)\diff t \\
	& = \int_{\R^d}(1-e^{-\varepsilon})_+p_n(t)\diff t = 1-e^{-\varepsilon}
\end{align*}
provided that $n$ is sufficiently large so that $B(0,4R)\subseteq \sqrt n(\Theta-\theta^*)$ for the last equality. For the second integral, we simply use the inequality $(1-x)_+\leq 1$ for all $x\geq 0$, which yields 
\begin{align*}
	\int_{\R^d\setminus B(0,4R)} & \left(1-e^{-(G_n(t)-\|t\|^2/2)}\mathds 1_{t\in\sqrt n(\Theta-\theta^*)}\right)_+p_n(t)\diff t \\
	& \quad\quad \leq \int_{\R^d\setminus B(0,4R)}p_n(t)\diff t \\
	& \quad\quad \leq (2\pi)^{-d/2}\int_{\R^d\setminus B(0,4R)}e^{-\frac{1}{2}(\|t\|^2-2R\|t\|)}\diff t \\
	& \quad\quad \leq (2\pi)^{-d/2}\int_{\R^d\setminus B(0,4R)}e^{-\frac{\|t\|^2}{4}}\diff t
\end{align*}
by Cauchy-Schwarz inequality, together with the fact that $\|Y_n\|\leq R$.
Since the last integral goes to $0$ as $R\to\infty$, we can choose $R$ sufficiently large so that the right-hand side of the last display is not larger than $\varepsilon$. 
Finally, assuming that $n$ is large enough (so $B(0,4R)\subseteq \sqrt n(\Theta-\theta^*)$), we have shown that on an event that holds with probability at least $1-2\varepsilon$ by the union bound, the integral from \eqref{eqn:goalTV} is bounded from above by $1-e^{-\varepsilon}+\varepsilon$. This concludes the proof.

\end{proof}

\section{Nearly misspecified case} \label{sec:case2}

Here, we assume that $\theta^*$ is on the boundary of $\Theta$ and that $\nabla\Phi(\theta^*)=0$. In this case, the constrained and unconstrained minimizers of $\Phi$ still coincide but the analysis of the posterior distribution requires further arguments. We still consider the following change of variable: $t=\sqrt n(\theta-\theta^*)$, whose posterior distribution $Q_n$ has a density $q_n$ with respect to the Lebesgue measure, which can be written as
$$q_n(t)=Z_nA_ne^{-G_n(t)-t^\top Y_n-\|Y_n\|^2/2}\mathds 1_{t\in\sqrt n(\Theta-\theta^*)}, \quad \forall t\in\R^d$$
where
\begin{itemize}
	\item $A_n$ is the (random) normalizing constant satisfying $A_n\int_{T}e^{-\frac{\|t+Y_n\|^2}{2}}\diff t=1$ -- recall that $T$ is the tangent cone of $\Theta$ at $\theta^*$ and, since $\Theta$ is assumed to have non-empty interior, $T$ also has non-empty interior;
	\item $Z_n$ is the (random) normalizing constant ensuring that $\int_{\R^d}q_n(t)\diff t=1$.
\end{itemize}

In the sequel, we denote by $T_n=\sqrt n(\Theta-\theta^*)$, so the sequence $(T_n)_{n\geq 1}$ is non-decreasing and satisfies $\bigcup_{n\geq 1}T_n=T$. 

The main result of this section is the following theorem. Given $a\in\R^d$ and a positive definite, symmetric matrix $A\in\R^{d\times d}$, we denote by $\mathcal N_d^T(a,A)$ the Gaussian distribution with mean $a$ and covariance matrix $A$, conditionned to stay in $T$. That is, its density with respect to the Lebesgue measure is proportional to $t\in\R^d\mapsto e^{-\frac{\|\Sigma^{-1/2}(t-a)\|^2}{2}}\mathds 1_{t\in T}$. 

\begin{theorem} \label{thm:main_case2}
Let Assumption~\ref{ass:subgradients_integr} hold. Assume that $\theta^*$ is on the boundary of $\Theta$ and that $\nabla\Phi(\theta^*)=0$. Then, 
	$$\TV(Q_n,\mathcal N_d^T(-\nabla^2\Phi(\theta^*)^{-1}Y_n,\nabla^2\Phi(\theta^*)^{-1}))\xrightarrow[n\to\infty]{}0$$
in probability.
\end{theorem}

\begin{remark}
	Note that in this case, it is no longer true that $\sqrt n(\hat \theta_n-\theta^*)+\nabla^2\Phi(\theta^*)^{-1}Y_n\xrightarrow[n\to\infty]{} 0$ in probability, where $\hat \theta_n$ is the constrained maximum likelihood estimator, as it was the case in Section~\ref{sec:case1} above, where the constrained and unconstrained maximum likelihood estimators were almost surely equal for sufficiently large $n$. Indeed, it follows from the proof of Theorem~7 in \cite{brunel2025asymptotics} that $\sqrt n(\hat \theta_n-\theta^*)-\pi_{T_n}^{\nabla^2\Phi(\theta^*)}(-Y_n)\xrightarrow[n\to\infty]{}0$ in probability, where $\pi_{T_n}^{\nabla^2\Phi(\theta^*)}$ is the metric projection onto $T_n$, with respect to the Euclidean structure defined by $\nabla^2\Phi(\theta^*)$. 
\end{remark}

\begin{proof}
For simplicity, we assume that $\nabla^2\Phi(\theta^*)=I_d$ again. As in the proof of Theorem~\ref{thm:main_case1}, the first step consists in showing that $Z_n\xrightarrow[n\to\infty]{}1$ in probability. To achieve this objective, it is necessary and sufficient to show that 
$$A_n\int_{\R^d}e^{-G_n(t)-t^\top Y_n-\|Y_n\|^2/2}\mathds 1_{t\in T_n}\diff t\xrightarrow[n\to\infty]{} 1$$
in probability, since $Z_n$ is the inverse of that quantity. Equivalently, let us show that 
$$A_n\left(\int_{\R^d}e^{-G_n(t)-t^\top Y_n-\|Y_n\|^2/2}\mathds 1_{t\in T_n}\diff t-\int_{\R^d}e^{-\frac{\|t+Y_n\|^2}{2}}\mathds 1_{t\in T}\diff t\right)\xrightarrow[n\to\infty]{}0$$
in probability. We proceed in two steps: First, we show that the sequence $(A_n)_{n\geq 1}$ is tight, and then, that the quantity inside the brackets in the previous display converges to $0$ in probability.

Fix $\varepsilon>0$ and $R>0$ such that $P(\|Y_n\|\leq R)\geq 1-\varepsilon$ for all $n\geq 1$, as in the previous section. Then, provided that $\|Y_n\|\leq R$, one can write the following inequalities:
\begin{equation*}
	1 = A_n\int_{T}e^{-\frac{\|t+Y_n\|^2}{2}}\diff t \geq A_n\int_T e^{-\|t\|^2-\|Y_n\|^2}\diff t \geq e^{-R^2} \int_T e^{-\|t\|^2}\diff t,
\end{equation*}
	yielding that $A_n\leq e^{R^2}\left(\int_T e^{-\|t\|^2}\diff t\right)^{-1}$. 
This shows tightness of $(A_n)_{n\geq 1}$. 

Now, let us show that 
\begin{equation} \label{eqn:goal2_1}
	\int_{\R^d}e^{-G_n(t)-t^\top Y_n-\|Y_n\|^2/2}\mathds 1_{t\in T_n}\diff t-\int_{\R^d}e^{-\frac{\|t+Y_n\|^2}{2}}\mathds 1_{t\in T}\diff t \xrightarrow[n\to\infty]{}0
\end{equation}
in probability.
	
Again, assume that $\|Y_n\|\leq R$. 
Since $T_n\subseteq T$ for all $n\geq 1$,
\begin{align*}
	\left|\int_{\R^d}e^{-\frac{\|t+Y_n\|^2}{2}}\mathds 1_{t\in T}\diff t-\int_{\R^d}e^{-\frac{\|t+Y_n\|^2}{2}}\mathds 1_{t\in T_n}\diff t \right| & = \int_{\R^d}e^{-\frac{\|t+Y_n\|^2}{2}}(\mathds 1_{t\in T}-\mathds 1_{t\in T_n})\diff t \\
	& \leq \int_{\R^d}e^{-\frac{\|t\|^2-2R\|t\|^2}{2}}(\mathds 1_{t\in T}-\mathds 1_{t\in T_n})\diff t
\end{align*}
where the last integral goes to $0$ as $n\to\infty$, by dominated convergence. This shows that the difference between the two integrals in the previous display converges to $0$ in probability, and \eqref{eqn:goal2_1} amounts in proving that 
\begin{equation} \label{eqn:goal3_1}
	\int_{T_n}\left(e^{-G_n(t)-t^\top Y_n-\|Y_n\|^2/2}-e^{-\frac{\|t+Y_n\|^2}{2}}\right)\diff t \xrightarrow[n\to\infty]{}0
\end{equation}
in probability.

Fix $R'>0$ and assume that $\sup_{\|t\|\leq R'}|G_n(t)-\|t\|^2/2|\leq \varepsilon$. By Lemma~\ref{lem:UnifConv}, this occurs with probability at least $1-\varepsilon$ if $n$ is chosen sufficiently large. Also assume that $\|Y_n\|\leq R$. Decompose the above integral as the sum of one integral $I_n$ on $T_n\cap B(0,R')$ and one integral $J_n$ on $T_n\setminus B(0,R')$. We bound $I_n$ and $J_n$ separately. First, write $I_n$ as 
\begin{equation*}
	I_n = \int_{B(0,R')} e^{-\|t+Y_n\|^2/2}\left(e^{-(G_n(t)-\|t\|^2/2)}-1\right)\mathds 1_{t\in T_n}\diff t
\end{equation*}
so, by the triangle inequality,
\begin{align*}
	|I_n| & \leq \int_{B(0,R')} e^{-\|t+Y_n\|^2/2}\left|e^{-(G_n(t)-\|t\|^2/2)}-1\right|\mathds 1_{t\in T_n}\diff t \\
	& \leq (e^\varepsilon-1)\int_{B(0,R')}e^{-\|t+Y_n\|^2/2}\mathds 1_{t\in T_n}\diff t \leq (e^\varepsilon-1)\int_{\R^d}e^{-\|t+Y_n\|^2/2}\diff t \\
	& = (e^\varepsilon-1)\int_{\R^d}e^{-\|t\|^2/2}\diff t = (2\pi)^{d/2}(e^\varepsilon-1)
\end{align*}
where we simply bounded $\mathds 1_{t\in T_n}$ by $1$ in the third inequality and performed a change of variable in the first equality. Second, for $J_n$, use \eqref{eqn:controlFn} (with $R'$ instead of $4R$) and use the triangle inequality to write
\begin{align*}
	|J_n| & \leq \int_{\R^d\setminus B(0,R')}e^{-\|t+Y_n\|^2/2}\left(e^{-(G_n(t)-\|t\|^2/2)}+1\right)\diff t \\
	& = \int_{\R^d\setminus B(0,R')}e^{-\|t+Y_n\|^2/2}\diff t + \int_{\R^d\setminus B(0,R')}e^{-G_n(t)-t^\top Y_n-\|Y_n\|^2/2}\diff t \\
	& \leq \int_{\R^d\setminus B(0,R')}e^{-\|t+Y_n\|^2/2}\diff t + \int_{\R^d\setminus B(0,R')}e^{-(R'/2-2\varepsilon/R')\|t\|-\varepsilon-t^\top Y_n-\|Y_n\|^2/2}\diff t \\
	& \leq \int_{\R^d\setminus B(0,R'-R)}e^{-\|u\|^2/2}\diff u + \int_{\R^d\setminus B(0,R')}e^{-(R'/2-2\varepsilon/R'-R)\|t\|-\varepsilon}\diff t
\end{align*}
where, for the first integral, we performed the change of variable $u=t+Y_n$, whose norm must be at least $R'-R$ and, for the second integral, we used Cauchy-Schwarz inequality. Since both these integrals go to $0$ as $R'\to\infty$, one can choose $R'$ such that $|J_n|\leq \varepsilon$. 

To recap, we proved that if $n$ is sufficiently large, with probability at least $1-2\varepsilon$ (that is, when $\|Y_n\|\leq R$ and $\sup_{t\in B(0,R')}|G_n(t)-\|t\|^2/2|\leq\varepsilon$), the integral in \eqref{eqn:goal3_1} is bounded by $(2\pi)^{d/2}(e^\varepsilon-1)+\varepsilon$ in absolute value. This proves \eqref{eqn:goal3_1}. Hence, we have obtained that $Z_n\xrightarrow[n\to\infty]{}1$ in probability. 

Now, we are ready to prove Theorem~\ref{thm:main_case1}, that is. Now that we have established that $Z_n\xrightarrow[n\to\infty]{} 1$ in probability, this amounts in showing that 
$$\int_{\R^d}\left|A_ne^{-G_n(t)-t^\top Y_n-\|Y_n\|^2/2}\mathds 1_{t\in T_n}-A_ne^{-\frac{\|t+Y_n\|^2}{2}}\mathds 1_{t\in T}\right|\diff t \xrightarrow[n\to\infty]{} 0$$
in probability. Moreover, since we have also shown that the sequence $(A_n)_{n\geq 1}$ is tight, we simply need to show that 
$$\int_{\R^d}\left|e^{-G_n(t)-t^\top Y_n-\|Y_n\|^2/2}\mathds 1_{t\in T_n}-e^{-\frac{\|t+Y_n\|^2}{2}}\mathds 1_{t\in T}\right|\diff t \xrightarrow[n\to\infty]{} 0$$
in probability, which follows directly from the previous computations. 
\end{proof}

\section{Misspecified case} \label{sec:case3}

Now, let us assume that $\theta^*$ is a boundary point and that $u:=\nabla\Phi(\theta^*)\neq 0$. For simplicity, we assume that $\theta^*=0$ and that $\nabla^2\Phi(\theta^*)=I_d$. This comes with no loss of generality, since one could simply perform a linear reparameterization of $\theta$. The main theorem of this section is stated in the general case (where $\nabla^2\Phi(\theta^*)$ is not necessarily the identity matrix) in Section~\ref{sec:case3_general} in the appendix (see Theorem~\ref{thm:case3_general}).

In this section, we assume that $\Theta$ can be written as a set of smooth and convex constraints:
$$\Theta=\{\theta\in\R^d:g_j(\theta)\leq 0, \forall j=1,\ldots,p\}$$
where $p\geq 1$ is an integer and $g_1,\ldots,g_p:\R^d\to\R$ are twice differentiable, convex functions. We let $J=\{j=1,\ldots,p:g_j(0)=0\}$ be the set of active constraints at $0$. In the following, we denote by $u_j=\nabla g_j(0)$ for all $j\in J$.

Then, the normal cone $N$ to $\Theta$ at $0$ is the conic hull of the $u_j$'s, $j\in J$, that is, $N=\{\sum_{j\in J}\lambda_j u_j:\lambda_j\geq 0,\forall j\in J\}$. 
The support cone $C$ of $\Theta$ at $0$ is the polar cone of $N$, that is, 
$$C=\{\theta\in \R^d:u_j^\top\theta\leq 0, \forall j\in J\}.$$
This is a convex polyhedron since it is defined by finitely many linear constraints. Note that some of these constraints may be redundant. In the sequel, we denote by $\tilde J$ the collection of all such $j\in J$ for which $C\cap u_j^\perp$ is a facet (i.e., a $(d-1)$-dimensional face) of $C$. 

By the first order condition, since $0$ is a minimizer of $\Phi$ on $\Theta$, $-u\in N$. Therefore, $C\cap u^\perp$ must be a face of $C$ (of any dimension between $0$ and $d-1$). That is, it is the intersection of finitely many facets of $C$. We denote by $J^*$ the collection of all $j\in \tilde J$ such that the facet $C\cap u_j^\perp$ contains the face $C\cap u^\perp$. That is, $-u$ is in the relative interior of the cone spanned by the $u_j$'s, $j\in J^*$, i.e., $-u=\sum_{j\in J^*}\lambda_j u_j$ for some positive coefficients $\lambda_j>0, j\in J^*$.

Now, we let $L=\textrm{span}(C\cap u^\top)=\{t\in \R^d:u_j^\top t=0, \forall j\in J^*\}$ be the linear hull of the face $C\cap u^\top$ of $C$. Its orthogonal space is given by $L^\perp=\textrm{span}(\{u_j:j\in J^*\})$. 

Then, we can rewrite $C\cap u^\top$ as
\begin{equation} \label{eqn:cone}
	C\cap u^\top = \{t\in L: u_j^\top t\leq 0, \forall j\in \tilde J, u_j^\top t=0, \forall j\in J^*\}.
\end{equation}
In particular, its relative interior is given by 
\begin{align} 
	\relint(C\cap u^\top) & = \{t\in L: u_j^\top t< 0, \forall j\in \tilde J\setminus J^*, u_j^\top t=0, \forall j\in J^*\} \nonumber \\
	& = \{t\in C\cap u^\perp:J(t)=J^*\} \label{eqn:relint_cone}
\end{align}
where, for all $t\in L$, we denote by $J(t)=\{j\in \tilde J:u_j^\top t=0\}$.

For $n\geq 1$, we let $C_2\pn=\{(t,s)\in L\times L^\perp:t/\sqrt n+s/n\in\Theta\}$ and we define the second order tangent set of $\Theta$ at $\theta^*=0$ (relative to $u$) as 
\begin{align*}
	C_2 & = \{(t,s)\in L\times L^\perp:u_j^\top t\leq 0, \forall j\in J, \frac{1}{2}t^\top\nabla^2 g_j(0)t+u_j^\top s\leq 0, \forall j\in J(t)\} \\
	& = \{(t,s)\in (C\cap u^\perp)\times L^\perp : s\in C_2(t)\}
\end{align*}
where $C_2(t)=\{s\in L^\perp: \frac{1}{2}t^\top\nabla^2 g_j(0)t+u_j^\top s\leq 0, \forall j\in J(t)\}$. We refer to \cite[Chapter 3]{bonnans2013perturbation} for related notions of tangent sets. Importantly, the set $C_2$ can be seen as the limiting set of the sequence of sets $C_2\pn, n\geq 1$, see Lemma~\ref{lemma:conv_sets} below (a precise statement could be made about the Painlevé-Kuratowsky convergence of $C_2\pn$ to $C_2$ but Lemma~\ref{lemma:conv_sets} will be sufficient to serve our purposes).

For convenience, we also let  $C_{2,0}\pn=\{(t,s)\in L\times L^\perp:t/\sqrt n+s/n\in\Theta_0\}$. For all $(t,s)\in C_{2,0}\pn$, denote by $G_n(t,s)=n(\Phi_n(t/\sqrt n+s/n)-\Phi_n(0))$ (recall that we have assumed that $\theta^*=0$). Since $\Theta_0$ is open and contains $0$, for any fixed $R>0$, $B_R:=\{(t,s)\in L\times L^\perp:\|t\|+\|s\|\leq R\}\subseteq C_{2,0}\pn$ for all sufficiently large $n$.

Now, we decompose every $\theta\in\Theta$ as $\theta=\frac{t}{\sqrt n}+\frac{s}{n}$ for some (uniquely defined) $(t,s)\in C_2\pn$, and we write the joint posterior density of $(t,s)$ as 
$$q_n(t,s)=\tilde Z_ne^{-n\left(\Phi_n(t/\sqrt n+s/n)-\Phi_n(0)\right)}\mathds 1_{(t,s)\in C_2\pn},$$
for all $(t,s)\in L\times L^\perp$, where $\tilde Z_n$ is a random, positive, normalizing constant. We denote by $Q_n$ the (random) distribution on $L\times L^\perp$ with density $q_n$ with respect to the Lebesgue measure of $L\times L^\perp$.\footnote{The Lebesgue measure of $L\times L^\perp$ is the product of the Lebesgue measures of $L$ and $L^\perp$.}

Denote by $R_n$ the probability measure on $L\times L^\perp$ whose density with respect to the Lebesgue measure is given by $r_n:(t,s)\in L\times L^\perp \mapsto A_ne^{-t^\top Y_n-\frac{\|t\|^2}{2}-s^\top u}\mathds 1_{(t,s)\in C_2}$, where $A_n>0$ is a (random) positive, normalizing constant satisfying 
\begin{equation} \label{eqn:norm_A}
	A_n\int_{L\times L^\perp}e^{-t^\top Y_n-\frac{\|t\|^2}{2}-s^\top u}\mathds 1_{(t,s)\in C_2}\diff t\diff s=1.
\end{equation}
Lemmas~\ref{lemma:bound_int} and~\ref{lemma:C2_convex} show, respectively, that the integral appearing in \eqref{eqn:norm_A} is almost surely finite and non-zero, yielding that $A_n$ is well defined. 
Then, we have the following theorem.

\begin{theorem} \label{THMCASE3}
Let Assumption~\ref{ass:subgradients_integr} hold. Assume that $\theta^*$ is on the boundary of $\Theta$ and that $\nabla\Phi(\theta^*)\neq 0$. Then, $\TV(Q_n,R_n)\xrightarrow[n\to\infty]{}0$ in probability.
\end{theorem}

\begin{proof}

First, write 
$$q_n(t,s)=Z_ne^{-\left(G_n(t,s)-t^\top Y_n-\frac{\|t\|^2}{2}-s^\top u\right)}A_ne^{-t^\top Y_n-\frac{1}{2}\|t\|^2-s^\top u}\mathds 1_{(t,s)\in C_2\pn},$$
for all $(t,s)\in L\times L^\perp$, where $Z_n$ is simply $\tilde Z_n/A_n$.

The strategy of the proof is similar as in the previous sections. First, we will show that $Z_n\xrightarrow[n\to\infty]{}1$ in probability, by showing that 
\begin{equation} \label{eqn:goal313}
	A_n\int_{L\times L^\perp} e^{-\left(G_n(t,s)-t^\top Y_n-\frac{\|t\|^2}{2}-s^\top u\right)}e^{-t^\top Y_n-\frac{\|t\|^2}{2}-s^\top u}\mathds 1_{(t,s)\in C_2\pn}\diff t\diff s \xrightarrow[n\to\infty]{} 1
\end{equation}
in probability. By a similar argument as in the proofs of Theorems\ref{thm:main_case1} and \ref{thm:main_case2}, the sequence $(A_n)_{n\geq 1}$ is tight, so in order to show \eqref{eqn:goal313}, it is sufficient to show that 
\begin{align} 
	& \int_{L\times L^\perp} e^{-t^\top Y_n-\frac{\|t\|^2}{2}-s^\top u}\left(e^{-\left(G_n(t,s)-t^\top Y_n-\frac{\|t\|^2}{2}-s^\top u\right)}\mathds 1_{(t,s)\in C_2\pn}-\mathds 1_{t\in C\cap u^\perp,s\in C_2(t)}\right)\diff t\diff s \nonumber \\
	& \hspace{12cm} \xrightarrow[n\to\infty]{} 0 \label{eqn:goal_case3}
\end{align}
in probability.

Fix $\varepsilon>0$. Since $(Y_n)_{n\geq 1}$ is tight by the central limit theorem, one can choose $r>0$ such that $P(\|Y_n\|\leq r)\geq 1-\varepsilon$ for all $n\geq 1$. Now, fix $R>0$. By \cite[Lemma 4]{brunel2025asymptotics}, we have that 
\begin{equation}
	\sup_{(t,s)\in B_R} \left|G_n(t,s)-t^\top Y_n-\frac{\|t\|^2}{2}-s^\top u\right| \xrightarrow[n\to\infty]{} 0
\end{equation}
in probability.
Assume that $\|Y_n\|\leq r$ and that $\DS \sup_{(t,s)\in B_R} \left|G_n(t,s)-t^\top Y_n-\frac{\|t\|^2}{2}-s^\top u\right|\leq \varepsilon$. These two events hold simultaneously with probability at least $1-2\varepsilon$, if $n$ is sufficiently large. 
Note that since $G_n$ is convex, we must have (by the same argument already used in previous proofs) that for all $(t,s)\notin B_R$, if we denote by $z=(t,s)$ and $\|z\|=\|t\|+\|s\|$, 
\begin{align}
	G_n(z) & \geq t^\top Y_n+\frac{R\|t\|^2}{2\|z\|}+s^\top u-\frac{\|z\|-R-1}{R}\varepsilon \nonumber \\
	& \geq -r\|t\| +\frac{R\|t\|^2}{2\|z\|}+s^\top u-\frac{\|z\|\varepsilon}{R}
\end{align}
(this inequality is obtained by writing $Rz/\|z\|$ as a convex combination of $0$ and $z$).
Moreover, by Lemmas~\ref{lemma:bound_uTs_finite} and~\ref{lemma:bound_uTs}, if $(t,s)\in C_2\pn$ or if $t\in C\cap u^\perp$ and $s\in C_2(t)$, it holds that $u^\top s\geq \alpha>0$ where $\alpha$ is a positive constant. Therefore, for all such $z=(t,s)\notin B_R$, we obtain
\begin{align*}
	G_n(z) & \geq -r\|t\| +\frac{R\|t\|^2}{2(\|t\|+\|s\|)}+\alpha\|s\|-\frac{\varepsilon}{R}(\|t\|+\|s\|) \\
	& = \frac{1}{\|t\|+\|s\|}\left((R/2-r-\varepsilon/R)\|t\|^2+(\alpha-r-2\varepsilon/R)\|t\|\|s\|+(\alpha-\varepsilon/R)\|s\|^2\right)
\end{align*}
which can be made non-negative irrespective of the values of $\|t\|,\|s\|$ by choosing $R$ large enough. 

Therefore, if we split the integral in \eqref{eqn:goal_case3} as the sum of an integral $I_n^{(1)}$ on $B_R$ and an integral $I_n^{(2)}$ on $(L\times L^\perp)\setminus B_R$, we have obtained 
$$|I_n^{(2)}| \leq 2 \int_{(L\times L^\perp)\setminus B_R} e^{r\|t\|-\frac{\|t\|^2}{2}-\alpha\|s\|}\diff t\diff s.$$
Since the map $(t,s)\in L\times L^\perp\mapsto e^{r\|t\|-\frac{\|t\|^2}{2}-\alpha\|s\|}$ is integrable on $L\times L^\perp$, the right-hand side of the display above can be made smaller than $\varepsilon$ by choosing $R$ large enough.

Now, we focus on $I_1^{(n)}$. First, write $I_1^{(n)}$ as
\begin{align*}
	I_1^{(n)} & = \int_{B_R} e^{-t^\top Y_n-\frac{\|t\|^2}{2}-s^\top u}\left(e^{-\left(G_n(t,s)-t^\top Y_n-\frac{\|t\|^2}{2}-s^\top u\right)}\mathds 1_{(t,s)\in C_2\pn}-\mathds 1_{t\in C\cap u^\perp,s\in C_2(t)}\right)\diff t\diff s \\
	& = \int_{B_R} e^{-t^\top Y_n-\frac{\|t\|^2}{2}-s^\top u}e^{-\left(G_n(t,s)-t^\top Y_n-\frac{\|t\|^2}{2}-s^\top u\right)}\left(\mathds 1_{(t,s)\in C_2\pn}-\mathds 1_{t\in C\cap u^\perp,s\in C_2(t)}\right)\diff t\diff s \\
	& \quad \quad \quad + \int_{B_R} e^{-t^\top Y_n-\frac{\|t\|^2}{2}-s^\top u}\left(e^{-\left(G_n(t,s)-t^\top Y_n-\frac{\|t\|^2}{2}-s^\top u\right)}-1\right)\mathds 1_{t\in C\cap u^\perp,s\in C_2(t)}\diff t\diff s \\
	& =: I_3^{(n)}+I_4^{(n)}.
\end{align*}
First, note that, by assumption, 
\begin{equation*}
	|I_3^{(n)}| \leq e^{\varepsilon}\int_{B_R} e^{r\|t\|-\frac{\|t\|^2}{2}-s^\top u}\left|\mathds 1_{(t,s)\in C_2\pn}-\mathds 1_{t\in C\cap u^\perp,s\in C_2(t)}\right|\diff t\diff s.
\end{equation*}

By Lemma~\ref{lemma:conv_sets} and dominated convergence, the right-hand side goes to $0$ as $n\to\infty$, so it becomes smaller than $\varepsilon$ for sufficiently large $n$.

\begin{align*}
	|I_4^{(n)}| & \leq (e^\varepsilon-1)\int_{B_R} e^{r\|t\|-\frac{\|t\|^2}{2}-s^\top u} 1_{t\in C\cap u^\perp,s\in C_2(t)}\diff t\diff s \\
	& \leq \alpha_2(e^\varepsilon-1)
\end{align*}
for some positive constant $\alpha_2$ which only depends on $R$ and $\Theta$. 

Wrapping up, we have proved \eqref{eqn:goal_case3}. Now, the rest of the proof of Theorem~\ref{THMCASE3} follows easily by writing 
\begin{equation*}
	\TV(Q_n,R_n) = \int_{L\times L^\perp}|q_n-r_n| = |1-1/Z_n| +\int_{L\times L^\perp}|q_n/Z_n-r_n|
\end{equation*}
and noting that the first term goes to $0$ in probability as shown above, and so does the second term, using tightness of $(A_n)_{n\geq 1}$ and the exact same arguments as for the proof of \eqref{eqn:goal_case3}.

\end{proof}

\section{Conclusion} \label{sec:conclusion}

In this work, we have established the asymptotic behavior of log-concave posterior distributions in two cases. First, for well-specified models, that is, when the true parameter belongs to the support of the prior distribution. Second, for misspecified models, when the true parameter does not belong to the support of the prior distribution. In the second case, we have shown that a non-classical scaling of the posterior distribution has a non-trivial limit that is supported on the second order tangent set of the support of the prior distribution. In all cases, we have only assumed that the log-likelihood function is concave in the parameter and positive definiteness of the Fisher information matrix. 
In the misspecified setup, we have assumed that the support $\Theta$ of the prior distribution can be written using finitely many twice differentiable, convex constraints. In fact, we believe that similar results can be achieved only assuming that the indicator function of $\Theta$ is twice epi-differentiable at $\theta^*$ along the direction $-\nabla\Phi(\theta^*)$,\footnote{This is equivalent to the projection on $\Theta$ having directional derivatives at $-\nabla\Phi(\theta^*)$, by \cite[Corollary 4.4]{do1992generalized}.} which would recover the assumption made in \cite{brunel2025asymptotics} for the asymptotic distribution of the MLE. This is left for future work. Remarkably, in the well-specified and nearly misspecified cases, we have made no assumptions on $\Theta$ -- other than being convex with non-empty interior. This is similar to the corresponding situations in the context of MLE \cite{brunel2025asymptotics}, where the indicator of $\Theta$ is always twice epi-differentiable alongat $\theta^*$ along the direction $-\nabla\Phi(\theta^*)=0$. 

\bibliographystyle{plain}
\bibliography{Biblio}

\appendix

\section{Proof of Theorem~\ref{THMPROPER}} \label{sec:proof_thm0}

By the law of large numbers, $\Phi_n(\theta)\xrightarrow[n\to\infty]{} \Phi(\theta)$ almost surely, for all fixed $\theta\in\Theta_0$. Therefore, by \cite[Corollary 1]{brunel2025asymptotics}, for all $r>0$ such that $B(\theta^*,r)\subseteq \Theta_0$, 
\begin{equation} \label{eqn:proper0}\sup_{\theta\in B(\theta^*,r)}|\Phi_n(\theta)-\Phi(\theta)|\xrightarrow[n\to\infty]{} 0
\end{equation}
almost surely. Let $\alpha>0$ be the smallest eigenvalue of $\nabla^2\Phi(\theta^*)$. Since $\Phi$ is twice differentiable in a neighborhood of $\theta^*$, there exists $r>0$ such that $B(\theta^*,r)\subseteq \Theta_0$ and for all $\theta\in B(\theta^*,r)$, all eigenvalues of $\nabla^2\Phi(\theta)$ are larger than $\alpha/2$. Hence, for all $\theta\in B(\theta^*,r)\cap \Theta$, 
\begin{equation} \label{eqn:proper1}
	\Phi(\theta)\geq \Phi(\theta_0)+\nabla\Phi(\theta^*)^\top(\theta-\theta^*)+\frac{\alpha}{4}\|\theta-\theta^*\|^2.
\end{equation}
Moreover, the first order optimality condition for $\theta^*$ yields that $u^\top(\theta-\theta^*)\geq 0$ for all $\theta\in\Theta$, so \eqref{eqn:proper1} becomes
\begin{equation} \label{eqn:proper1}
	\Phi(\theta)\geq \Phi(\theta_0)+\frac{\alpha}{4}\|\theta-\theta^*\|^2,
\end{equation}
for all $\theta\in B(\theta^*,r)\cap \Theta$. 
Let $\varepsilon=\alpha r^2/16$. By \eqref{eqn:proper0}, with probability $1$, it holds for all sufficiently large $n\geq 1$ that, simultaneously for all $\theta\in B(\theta^*,r)\cap \Theta$ 
\begin{equation} \label{eqn:proper2}
	\Phi_n(\theta) \geq \Phi(\theta)-\varepsilon \geq \Phi(\theta^*)-\varepsilon+\frac{\alpha}{8}\|\theta-\theta^*\|^2.
\end{equation}
and 
\begin{equation} \label{eqn:proper2}
	\Phi_n(\theta^*) \leq \Phi(\theta^*)+\varepsilon.
\end{equation}
In particular, for all $\theta\in B(\theta^*,r)\cap\Theta$, 
\begin{equation} \label{eqn:proper21}
	\Phi_n(\theta)-\Phi_n(\theta^*)\geq \alpha r^2/8.
\end{equation}

If there is no such $\theta\in\Theta$ with $\|\theta-\theta^*\|>r$, then $\Theta$ is compact and Theorem~\ref{THMPROPER} is trivial. Assume otherwise, and fix $\theta\in\Theta$ with $\|\theta-\theta^*\|>r$. Let $\theta_1=\theta^*+r(\theta-\theta^*)/\|\theta-\theta^*\|$, which is the intersection of $\R_+(\theta-\theta^*)$ with the boundary of $B(\theta^*,r)$. By convexity of $\Theta$, $\theta_1\in\Theta$ and convexity of $\Phi_n$ yields that
$$\Phi_n(\theta_1)\leq \Phi_n(\theta^*)+\frac{r}{\|\theta-\theta^*\|}(\Phi_n(\theta)-\Phi_n(\theta^*)),$$
which can be written as 
\begin{equation}
	\Phi_n(\theta)\geq \frac{\|\theta-\theta^*\|}{r}(\Phi_n(\theta_1)-\Phi_n(\theta^*))+\Phi_n(\theta^*).
\end{equation}
Combined with \eqref{eqn:proper21}, we have obtained that, with probability $1$, for all sufficiently large $n$, it holds that 
\begin{equation*}
	\Phi_n(\theta)\geq \frac{\alpha r\|\theta-\theta^*\|}{8}+\Phi_n(\theta^*), \quad \forall \theta\in \Theta\setminus B(\theta^*,r).
\end{equation*}
Thus, for all $\theta\in \Theta_0\setminus B(\theta^*,r)$, 
$$e^{-n\Phi_n(\theta)}\mathds 1_{\theta\in\Theta} \leq e^{-\frac{n\alpha r\|\theta-\theta^*\|}{8}-\Phi_n(\theta^*)}$$
and, hence, the map is integrable on $\Theta_0\setminus B(\theta^*,r)$ and, therefore, on $\Theta_0$, since it is continuous on $B(\theta^*,r)$ by convexity of $\Phi_n$. This ends the proof of the theorem.

\begin{remark}
	Here, we have proved Theorem~\ref{THMPROPER} in the specific case when $V=\delta_\Theta$. In fact, one can easily check that our proof goes through as long as $V$ is bounded from below. This roughly means that the prior distribution is not too informative, i.e., it does not put too much mass on any specific region of its support. 
\end{remark}

\section{General statement of Theorem~\ref{THMCASE3}} \label{sec:case3_general}

In Section~\ref{sec:case3}, we only stated Theorem~\ref{THMCASE3} in the case when the Fisher information is the identity matrix, i.e., $\nabla^2\Phi(\theta^*)=I_d$. This simplification does not come at a loss of generality since the general case can be recovered by a simple affine transformation. However, the general version requires some additional notation, which we introduce here. 

Denote by $S=\nabla^2\Phi(\theta^*)$ the Fisher information matrix. By assumption, it is positive definite, hence, it equips $\R^d$ with a Euclidean structure. We denote by $\langle x,y\rangle_{S}:=x^\top Sy$ for all $x,y\in\R^d$ the corresponding scalar product. The associated Euclidean norm is denoted by $\|\cdot\|_S$. That is, $\|x\|_S=(x^\top Sx)^{1/2}$.
For a non-empty subset $A\subseteq \R^d$, we denote by $A^{\perp_S}$ the orthogonal subspace to $A$ with respect to $\langle\cdot,\cdot\rangle_S$. That is, 
$$A^{\perp_S}=\{y\in\R^d:\langle x,y\rangle_S = 0, \forall x\in A\}.$$

\begin{theorem} \label{thm:case3_general}
	Let Assumption~\ref{ass:subgradients_integr} hold. Let $\Theta$ be written as $\{\theta\in\R^d=g_j(\theta)\leq 0, j=1,\ldots,p\}$ where $p\geq 1$ is some integer and $g_1,\ldots,g_p:\R^d\to\R$ are convex and twice differentiable. Assume that $\Theta$ has non-empty interior. 
	
Further assume that $\theta^*$ is on the boundary of $\Theta$ and that $\nabla\Phi(\theta^*)\neq 0$ and let $J=\{j=1,\ldots,p: g_j(\theta^*)=0\}$. Let $C$ be the support cone of $\Theta$ at $\theta^*$. 

Denote by $u=\nabla\Phi(\theta^*)$ and $S=\nabla^2\Phi(\theta^*)$. Assume that $S$ is positive definite. 

Let $\tilde J$ be the collection of all $j\in J$ such that $C\cap u_j^\top$ is a facet of $C$ and $J^*$ be the collection of those $j\in \tilde J$ corresponding to a facet that contains the face $C\cap u^\perp$. 

Let $L=\textrm{span}(C\cap \nabla\Phi(\theta^*)^\top)$ and $C_2=\{(t,s)\in L\times L^{\perp_S}:t\in C\cap u^\perp, \frac{1}{2}t^\top \nabla^2 g_j(\theta^*) t+u_j^\top s\leq 0, \forall j\in J(t)\}$ where $J(t)=\{j\in \tilde J:u_j^\top t=0\}$. 

Then, by writing $\theta=\frac{t}{\sqrt n}+\frac{s}{n}$ for a unique $(t,s)\in L\times L^{\perp_S}$, the posterior distribution $Q_n$ of $(t,s)$ satisfies
$$\TV(Q_n,R_n)\xrightarrow[n\to\infty]{} 0$$
in probability, where $R_n$ is the probability distribution on $L\times L^{\perp_S}$ whose density with respect to the Lebesgue measure is proportional to 
$$(t,s)\in L\times L^{\perp} \mapsto e^{-t^\top Y_n-\frac{\|t\|_S^2}{2}-s^\top u}\mathds 1_{(t,s)\in C_2}.$$

\end{theorem}

\section{Intermediate lemmas}

\begin{lemma} \label{lemma:bound_uTs_finite}
There exists $\alpha>0$ such that for all $t\in C\cap u^\top$ and for all $s\in C_2(t)$, $u^\top s\geq \alpha\|s\|$.
\end{lemma}

\begin{proof}
	First, recall that for all $t\in C\cap u^\perp$, the set $J(t)$ contains $J^*$. Moreover, since the $g_j$'s are convex, $t^\top\nabla^2 g_j(0)t$ is always non-negative. Hence, for all $t\in C\cap u^\perp$, $C_2(t)\subseteq \{s\in L^perp: s^\top u_j\leq 0, \forall j\in J^*\}$. 
	
Fix some $s\in L^\perp$ such that $s^\top u_j\leq 0$ for all $j\in J^*$. Assume that $s^\top u=0$. Then,
\begin{equation*}
	0 = -s^\top u = \sum_{j\in J^*}\lambda_j s^\top u_j
\end{equation*}
which is a sum of non-negative terms. Hence, all these terms must vanish, yielding that $s^\top u_j=0$ for all $j\in J^*$. That is, $s\in L$ and, therefore, $s$ must be $0$. We have established that in the closed cone $\{s\in L^\perp:s^\top u_j\leq 0, \forall j\in J^*\}$, $u^\top s>0$ as soon as $s\neq 0$. 
By a compacity argument, we obtain the existence of $\alpha>0$ such that $u^\top s\geq \alpha$ for all $s\in L^\perp$ with $s^\top u_j\leq 0, \forall j\in J^*$ and $\|s\|=1$. This yields the result by homogeneity.	
\end{proof}

We also have the following result. 

\begin{lemma} \label{lemma:bound_uTs}
Let $K=\{s\in L^\perp:\exists t\in C\cap u^\perp, t+s\in\Theta\}$. Then, there exists $\alpha>0$ such that for all $s\in K$, $u^\top s\geq \alpha \|s\|$. 
\end{lemma}

\begin{proof}
Since the $u_j$'s, $j\in J^*$, are normal vectors to $\Theta$ at $0$, it holds that $\Theta\subseteq \{\theta\in\R^d:u_j^\top\theta\leq 0, \forall j\in J^*\}=:\Theta^*$. Therefore, $\DS \inf\{u^\top s/\|s\|: \exists t\in C\cap u^\perp, t+s\in \Theta\} \geq \inf\{u^\top s/\|s\|: \exists t\in C\cap u^\perp, t+s\in \Theta^*\}$. For all $t\in C$, $u_j^\top t=0$ for all $j\in J^*$, since $N$ and $C$ are polar to each other. If, in addition, $t\in u^\perp$, then $\sum_{j\in J^*}\lambda_ju_j^\top t=0$, so each term in that sum must be $0$, i.e., $u_j^\top t=0$ for each $j\in J^*$. We thus obtain that $\DS \inf\{u^\top s/\|s\|: \exists t\in C\cap u^\perp, t+c\in \Theta\} \geq \inf\{u^\top s/\|s\|: s\in L^\perp, u_j^\top s\leq 0, \forall j\in J^*\}$. 
Using the same argument as in the proof of Lemma~\ref{lemma:bound_uTs_finite}, we have that $u^\top s>0$ for all $s\in L^\perp$ with $u_j^\top s\leq 0, \forall j\in J^*$. Hence, by a compacity argument, there must exist $\alpha>0$ such that $u^\top s\geq\alpha$ for all $s\in L^\perp$ with $u_j^\top s\leq 0, \forall j\in J^*$ and $\|s\|=1$. This yields the desired result. 
\end{proof}

\begin{lemma} \label{lemma:bound_int}
	There exists $a_1>0$ such that for all $t\in C\cap u^\perp$, 
	$$\int_{L^\perp} e^{-s^\top u}\mathds 1_{s\in C_2(t)}\diff t\leq a_1.$$
\end{lemma}

\begin{proof}
By Lemma~\ref{lemma:bound_uTs_finite},  there is some $\alpha>0$, independent of $t\in C\cap u^\perp$ such that 
\begin{align*}
	\int_{L^\perp} e^{-s^\top u}\mathds 1_{s\in C_2(t)}\diff s & \leq \int_{L^\perp} e^{-\alpha\|s\|}\mathds 1_{s\in C_2(t)}\diff s \\
	& \leq \int_{L^\perp} e^{-\alpha\|s\|}\diff s <\infty.
\end{align*}

\end{proof}

For the next lemma, recall that $J^*$ is the set of indices $j\in J$ such that $C\cap u_j^\perp$ is a facet of $C$ containing the face $C\cap u^\top$. 

\begin{lemma} \label{lemma:C2_convex}
	The set $C_2$ is convex. Its interior is non-empty and it is given by
\begin{equation} \label{eqn:relint_C}
	\inter(C_2)=\{(t,s)\in L\times L^\perp: t\in \relint(C\cap u^\perp), \frac{1}{2}t^\top\nabla^2 g_j(0)t+u_j^\top s<0, \forall j\in J^*\}.
\end{equation}
\end{lemma}

\begin{proof}
\item
\paragraph{Convexity of $C_2$.} 
Let us first check that $C_2$ is a convex set. Let $(t_0,s_0),(t_1,s_1)$ be two elements of $C_2$ and $\lambda\in (0,1)$. Denote by $t_\lambda=(1-\lambda)t_0+\lambda t_1$ and $s_\lambda=(1-\lambda)s_0+\lambda s_1$. Our goal is to show that $(t_\lambda,s_\lambda)\in C_2$. First, $t_\lambda\in C\cap u^\perp$ by convexity of the set $C\cap u^\perp$. Moreover, it is easy to check that $J(t_\lambda)\subseteq J(t_0)\cap J(t_1)$. Now, we simply need to check that $s_\lambda\in C_2(t_\lambda)$, that is, 
\begin{equation} \label{eqn:C2_convex_check}
	\frac{1}{2}t^\top\nabla g_j(0)t+u_j^\top s\leq 0, \quad \forall j\in J(t_\lambda),
\end{equation}
when $(t,s)=(t_\lambda,s_\lambda)$. Since \eqref{eqn:C2_convex_check} is satisfied by both $(t_0,s_0)$ and $(t_1,s_1)$ and the maps $(t,s)\mapsto t^\top\nabla^2 g_j(0)t+u_j^\top s$ are convex, \eqref{eqn:C2_convex_check} is automatically satisfied by $(t_\lambda,s_\lambda)$, which yields the result. 

\item
\paragraph{Validity of \eqref{eqn:relint_C}.} 

Now, let us check \eqref{eqn:relint_C}. The right-left inclusion is obvious, so let us only check the left-right one. 
Let $(t,s)\in \inter(C_2)$. Necessarily, $t$ must be in $\relint(C\cap u^\top)$. Indeed, if this was not the case, there would be some $j\in \tilde J\setminus J^*$ such that $u_j^\top t=0$. Since $j\notin J^*$, there must be some $z\in L$ with $u_j^\top z\neq 0$. Without loss of generality, assume that $u_j^\top z>0$. Then, for all $\varepsilon>0$, $t+\varepsilon z\notin C\cap u^\top$, so $(t+\varepsilon z,s)$ cannot be in $C_2$ and $(t,s)$ cannot be an interior point of $C_2$. 

Now, let us assume, for the sake of contradiction, that $\frac{1}{2}t^\top \nabla^2 g_j(0)t+u_j^\top s=0$ for some $j\in J^*$. Note that by definition of $J^*$, $L\subseteq u_j^\perp$. Moreover, $u_j\neq 0$, also by definition of $J^*$, since $C\cap u_j^\perp$ is assumed to be a facet of $C$ -- so it cannot be the whole $C$. Hence, $u_j\notin L$. That is, there exists $y\in L^\perp$ with $u_j^\top y\neq 0$. Again with no loss of generality, we assume that $u_j^\top y>0$. Hence, for all $\varepsilon>0$, $\frac{1}{2}t^\top \nabla^2 g_j(0)t+u_j^\top (s+\varepsilon y)>0$ so $(t,s+\varepsilon y)$ cannot be in $C_2$. 

\item
\paragraph{Non-empty interior of $C_2$.} 

Assume, for the sake of contradiction, that $\inter(C_2)$ is empty. Since the relative interior of any non-empty convex set is always non-empty, $\relint(C\cap u^\top)\neq\emptyset$. Hence, for all $t\in\relint(C\cap u^\perp)$ and for all $s\in L^\perp$, there must exist $j\in J^*$ such that $\frac{1}{2}t^\top \nabla^2 g_j(0)t+u_j^\top s\geq 0$. That is, for all $s\in L^\perp$, we have:
$$\max_{j\in J^*}\left(\frac{1}{2}t^\top \nabla^2 g_j(0)t+u_j^\top s\right), \quad \forall t\in \relint(C\cap u^\perp).$$
By letting $t$ go to $0$ in the previous display, we obtain that for all $s\in L^\perp$, there must exist some $j\in J^*$ such that $u_j^\top s\geq 0$. In other words, the set $\{s\in L^\perp:u_j^\top s <0, \forall j\in J^*\}$ must be empty. This is the relative interior of the convex polyhedron $P:=\{s\in L^\perp:u_j^\top s \leq 0, \forall j\in J^*\}$ in $L^\perp$. Therefore, the convex set $L\times P$ must have empty interior. However, this convex set contains $\{(t,s)\in L\times L^\perp:t+s\in\Theta\}$, which is isometric to $\Theta$ and $\Theta$ is assumed to have non-empty interior. This yields a contradiction. 

\end{proof}

\begin{lemma} \label{lemma:conv_sets}
	For almost all\footnote{with respect to the Lebesgue measure of $L\times L^\perp$} $(t,s)\in L\times L^\perp$, 
	$$\mathds 1_{(t,s)\in C_2\pn} \xrightarrow[n\to\infty]{} \mathds 1_{(t,s)\in C_2}.$$
\end{lemma}

\begin{proof}

Recall that $J$ is the set of indices of active constraints at $0$, i.e., $J=\{j=1,\ldots,p:g_j(0)=0\}$. We have defined $\tilde J$ as the collection of all such $j\in J$ for which $C\cap u_j^\perp$ is a facet of $C$ and $J^*\subseteq J$ as the collection of such $j$'s for which $C\cap u_j^\perp$ is a facet of $C$ containing the face $C\cap u^\perp$.
Therefore, one can write
$$C\cap u^\perp=\{t\in L: u_j^\top t\leq 0, \forall j\in \tilde J, u_j^\top t=0, \forall j\in J^*\}.$$
Therefore the relative interior of $C\cap u^\perp$ in $L$ is 
$$\relint(C\cap u^\perp)=\{t\in L: u_j^\top t < 0, \forall j\in \tilde J\setminus J^*, u_j^\top t, \forall j\in J^*\},$$
that is, $\relint(C\cap u^\perp)=\{t\in C\cap u^\perp:J(t)=J^*\}$ where we recall that $J(t)=\{j\in \tilde J:u_j^\top t=0\}$.

Fix $(t,s)\in L\times L^\perp$. Recall that $(t,s)\in C_2\pn$ if and only if $t/\sqrt n+s/n\in \Theta$. Since the functions $g_j, j=1,\ldots,p$, defining $\Theta$ are convex, they must be continuous on $\R^d$. Hence, for all large enough $n$ (how large -- this depends on $(t,s)$ which is fixed), $t/\sqrt n+s/n\in \Theta$ if and only if $g_j(t/\sqrt n+s/n)\leq 0$ for all $j\in J$. Indeed, for $j\notin J$, $g_j(0)<0$ by definition, so $g_j(\theta)$ remains negative for all $\theta$ in a vicinity of the origin. Therefore, for large enough $n$, 
$$\mathds 1_{(t,s)\in C_2\pn}=\mathds 1_{g_j(t/\sqrt n+s/n)\leq 0, \forall j\in J}.$$
Recall the definition of $J(t)=\{j\in J:u_j^\top t=0\}$.
Fix $j\in J$. If $j\in J(t)$, write 
\begin{equation} \label{eqn:exp_gj_0}
	g_j(t/\sqrt n+s/n) = \frac{1}{n}\left(\frac{1}{2}t^\top \nabla^2 g_j(0)t+u_j^\top s+\varepsilon_n^{(j)}\right)
\end{equation}
for some $\varepsilon_n^{(j)}\xrightarrow[n\to\infty]{}0$. If $j\notin J(t)$, write 
\begin{equation} \label{eqn:exp_gj_1}
	g_j(t/\sqrt n+s/n) = \frac{1}{\sqrt n}\left(u_j^\top t+\eta_n^{(j)}\right)
\end{equation}
for some $\eta_n^{(j)}\xrightarrow[n\to\infty]{}0$. Note that $\varepsilon_n^{(j)}$ and $\eta_n^{(j)}$ depend on $s$ and $t$, but these are fixed here.

Therefore, for all large enough $n$,
$$(t,s)\in C_2\pn \iff \begin{cases} u_j^\top t+\eta_n^{(j)}\leq 0, \forall j\in J\setminus J(t) \\ \frac{1}{2}t^\top\nabla^2 g_j(0)t+u_j^\top s+\varepsilon_n^{(j)}\leq 0, \forall j\in J(t).\end{cases}$$

Assume that $(t,s)\in\inter(C_2)$. That is, by Lemma~\ref{lemma:C2_convex}, $t\in \relint(C\cap u^\top)$ and $s\in\relint(C_2(t))$. That is, by \eqref{eqn:relint_cone}, $t\in C\cap u^\perp$ with $J(t)=J^*$. Hence, for all $j\in J\setminus J^*$, $u_j^\top t<0$. Since the $\eta_n^{(j)}, j\notin J^*$, go to $0$, it must hold that $u_j^\top t+\eta_j^{(n)}\leq 0$ for all $j\notin J^*$, for all sufficiently large $n$. Moreover, since $s\in \relint(C_2(t))$, $\frac{1}{2}t^\top\nabla^2 g_j(0)t+u_j^\top s<0$ for all $j\in J^*$. Then, since $\varepsilon_n^{(j)}\xrightarrow[n\to \infty]{} 0$ for all $j\in J(t)=J^*$, it must be that $\frac{1}{2}t^\top\nabla^2 g_j(0)t+u_j^\top s+\varepsilon_n^{(j)}\leq 0$ for all $j\in J(t)$, for all sufficiently large $n$. 
We thus have shown that if $(t,s)\in\relint(C_2)$, then $\mathds 1_{(t,s)\in C_2\pn}=1=\mathds 1_{(t,s)\in C_2}$ for all sufficiently large $n$. 

Now, let $(t,s)\in L\times L^\perp$ with $(t,s)\notin C_2$. That is, $t\notin C\cap u^\top$, or $t\in C\cap u^\top$ but $\frac{1}{2}t^\top\nabla^2 g_j(0)t+u_j^\top s>0$ for some $j\in J(t)$. Similarly as above, we obtain that $\mathds 1_{(t,s)\in C_2\pn}=0=\mathds 1_{(t,s)\in C_2}$ for all sufficiently large $n$. 

Finally, since $C_2$ is a convex set by Lemma~\ref{lemma:C2_convex}, the Lebesgue measure of its boundary is zero, so we have proven that for almost all $(t,s)\in L\times L^\perp$, 
$$\mathds 1_{(t,s)\in C_2\pn} \xrightarrow[n\to\infty]{}\mathds 1_{(t,s)\in C_2}.$$ 

\end{proof}

\end{document}